\documentclass[12pt, reqno]{amsart}
\usepackage{amsmath, amsthm, amscd, amsfonts, amssymb, graphicx, color}
\usepackage[bookmarksnumbered, colorlinks, plainpages]{hyperref}

\input{mathrsfs.sty}

\newtheorem{theorem}{Theorem}[section]
\newtheorem{lemma}[theorem]{Lemma}
\newtheorem{proposition}[theorem]{Proposition}
\newtheorem{corollary}[theorem]{Corollary}
\theoremstyle{definition}
\newtheorem{definition}[theorem]{Definition}
\newtheorem{example}[theorem]{Example}

\theoremstyle{remark}
\newtheorem{remark}[theorem]{Remark}
\numberwithin{equation}{section}

\begin{document}

\title{Chebyshev type inequalities for Hilbert space operators}

\author[M.S. Moslehian, M. Bakherad]{Mohammad Sal Moslehian$^1$ and Mojtaba Bakherad$^2$}

\address{$^1$ Department of Pure Mathematics, Center of Excellence in
Analysis on Algebraic Structures (CEAAS), Ferdowsi University of
Mashhad, P. O. Box 1159, Mashhad 91775, Iran}
\email{moslehian@um.ac.ir,
moslehian@member.ams.org}

\address{$^{2}$ Department of Pure Mathematics, Ferdowsi University of Mashhad, P. O. Box 1159, Mashhad 91775, Iran.}
\email{mojtaba.bakherad@yahoo.com; bakherad@member.ams.org}

\subjclass[2010]{Primary 47A63, Secondary 47A60, 46L05.}

\keywords{Chebyshev inequality; Hadamard product; Bochner integral;
super-multiplicative function; singular value; operator mean.}

\begin{abstract}
We establish several operator extensions of the Chebyshev
inequality. The main version deals with the Hadamard product of
Hilbert space operators. More precisely, we prove that if
$\mathscr{A}$ is a $C^*$-algebra, $T$ is a compact Hausdorff space
equipped with a Radon measure $\mu$, $\alpha: T\rightarrow [0,
+\infty)$ is a measurable function and $(A_t)_{t\in T}, (B_t)_{t\in
T}$ are suitable continuous fields of operators in ${\mathscr A}$
having the synchronous Hadamard property, then
\begin{align*}
\int_{T} \alpha(s) d\mu(s)\int_{T}\alpha(t)(A_t\circ B_t) d\mu(t)\geq\left(\int_{T}\alpha(t) A_t d\mu(t)\right)\circ\left(\int_{T}\alpha(s) B_s d\mu(s)\right).
 \end{align*}
We apply states on $C^*$-algebras to obtain some versions related to
synchronous functions. We also present some Chebyshev type
inequalities involving the singular values of positive $n\times n$
matrices. Several applications are given as well.
\end{abstract} \maketitle
\section{Introduction and preliminaries}

Let ${\mathbb B}({\mathscr H})$ denote the $C^*$-algebra of all
bounded linear operators on a complex Hilbert space ${\mathscr H}$ together with the operator norm $\|\cdot\|$. Let $I$ stand for the identity operator. In the case when ${\rm dim}{\mathscr H}=n$, we identify ${\mathbb
B}({\mathscr H})$ with the matrix algebra $\mathbb{M}_n$ of all
$n\times n$ matrices with entries in the complex field $\mathbb{C}$.
An operator $A\in{\mathbb B}({\mathscr H})$ is called positive
(positive semidefinite for a matrix $A$) if $\langle Ax,x\rangle\geq0$
for all $x\in{\mathscr H }$ and then we write $A\geq0$. By a
strictly positive operator (positive definite for a matrix) $A$,
denoted by $A>0$, we mean a positive invertible operator. For
self-adjoint operators $A, B\in{\mathbb B}({\mathscr H})$, we say
$B\geq A$ ($B>A$, resp.) if $B-A\geq0$ ($B-A>0$, resp.). For $A\in\mathbb{M}_n$, the singular values of
$A$, denoted by $s_1(A), s_2(A), \cdots, s_n(A)$, are the eigenvalues
of the positive matrix $|A|=(A^*A)^{1\over2}$ enumerated as $s_1(A)
\geq\cdots\geq s_n(A)$ with their multiplicities counted.

The Gelfand map $f(t)\mapsto f(A)$ is an
isometrically $*$-isomorphism between the $C^*$-algebra
$C({\rm sp}(A))$ of continuous functions on the spectrum ${\rm sp}(A)$
of a self-adjoint operator $A$ and the $C^*$-algebra generated by $I$ and $A$. If $f, g\in C({\rm sp}(A))$, then
$f(t)\geq g(t)\,\,(t\in{\rm sp}(A))$ implies that $f(A)\geq g(A)$. If $f$ be a continuous real valued function on an interval $J$. The
function $f$ is called operator monotone (operator decreasing, resp.) if
$A\leq B$ implies $f(A)\leq f(B)$ ($f(B)\leq f(A)$, resp.) for all
$A, B\in {\mathbb B}_h^J({\mathscr H})$, where ${\mathbb B}_h^J({\mathscr H})$ is the set of all
self-adjoint operators in ${\mathbb B}({\mathscr H})$, whose spectra
are contained in $J$; cf. \cite{naj}.

Given an orthonormal basis $\{e_j\}$ of a Hilbert space ${\mathscr H}$, the Hadamard product $A \circ B$ of two operators $A, B \in {\mathbb B}({\mathscr H})$ is defined by $\langle A\circ B e_i, e_j\rangle = \langle Ae_i, e_j\rangle \langle Be_i, e_j\rangle$. Clearly $A \circ B=B \circ A$. It is known that the Hadamard
product can be presented by filtering the
tensor product $A \otimes B$ through a positive linear map. In fact,
\begin{align}\label{nn1}
 A\circ B=U^*(A\otimes B)U,
\end{align}
where $U:{\mathscr H}\to {\mathscr H}\otimes{\mathscr H}$ is the isometry
defined by $Ue_j=e_j\otimes e_j$; see \cite{FUJ}. It follows from \eqref{nn1} that if $A\geq 0$ and $B\geq 0$, then 
\begin{align}\label{nn2}
A\circ B\geq 0.
\end{align}
For matrices, one easily observe \cite{STY} that the Hadamard product of $A=(a_{ij})$ and
$B=(b_{ij})$ is $A\circ B=(a_{ij}b_{ij})$, a principal
submatrix of the tensor product $A\otimes B=(a_{ij}B)_{1 \leq i,j\leq n}$. From now on when we deal with the Hadamard product of operators, we explicitly assume that an orthonormal basis is fixed.

The axiomatic theory of operator means has been developed by Kubo and Ando \cite{ando}. An operator mean is a binary operation $\sigma$ defined on the set of strictly positive operators, if the following conditions hold:\\
 $ (\textrm{i})\hspace{.1cm}A\leq C, B\leq D $ \hspace{.2cm}imply
 $A\sigma B\leq C\sigma D ;\\$
 $(\textrm{ii}) \hspace{.1cm}A_n\downarrow A, B_n\downarrow B$ \hspace{.2cm}imply
 $A_n\sigma B_n\downarrow A\sigma B$, where $A_n\downarrow A$ means that $A_1\geq A_2\geq \cdots$ and $A_n\rightarrow A$ as $n\rightarrow\infty$ in the strong operator topology; \\
 $ (\textrm{iii})\hspace{.1cm}T^*(A\sigma B)T\leq (T^*AT)\sigma (T^*BT)\hspace{.5cm}(T\in{\mathbb B}({\mathscr H}));$\\
 $(\textrm{iv}) \hspace{.1cm}I\sigma I=I$ .
 
There exists an affine order isomorphism between the class of
operator means and the class of positive operator monotone functions
$f$ defined on $(0,\infty)$ with $f(1)=1$ via
$f(t)I=I\sigma(tI)\hspace{.1cm}(t>0)$. In addition, $A\sigma
B=A^{1\over 2}f(A^{-1\over2}BA^{-1\over2})A^{1\over2}$ for all
strictly positive operators $A, B$. The operator monotone function
$f$ is called the representing function of $\sigma$. Using a limit
argument by $A_\varepsilon=A+\varepsilon I$, one can extend the
definition of $A\sigma B$ to positive operators. The operator means
corresponding to the operator monotone functions
$f_{\sharp_\mu}(t)=t^\mu$ and $f_{!}(t)=\frac{2t}{1+t}$ on
$[0,\infty)$ are the operator weighted geometric mean $A\sharp_\mu
B=A^{\frac{1}{2}}\left(A^{\frac{-1}{2}}BA^{\frac{-1}{2}}\right)^{\mu}A^{\frac{1}{2}}$
and the operator harmonic mean $A!B=2(A^{-1}+B^{-1})^{-1}$,
respectively.

Let us consider the real sequences $a=(a_1, \cdots, a_n)$, $b=(b_1,
\cdots, b_n)$ and a non-negative sequence $w=(w_1, \cdots, w_n)$.
Then the weighed Chebyshev function is defined by
 \begin{align*}
 T(w; a, b):=\sum_{j=1}^nw_{j}\sum_{j=1}^nw_ja_jb_j
 -\sum_{j=1}^nw_ja_j\sum_{j=1}^nw_jb_j.
 \end{align*}
In 1882, Chebyshev \cite{che1} proved that if $a$ and $b$ are
monotone in the same sense, then
$T(w; a, b)\geq 0$. Some integral generalizations of this inequality was given by Barza,
Persson and Soria \cite{BPS}. The Chebyshev inequality is a complement of the Gr\" uss inequality; see \cite{MR} and references therein.
 
A related notion is synchronicity. Two continuous functions $f, g:J\rightarrow \mathbb{R}$
are called synchronous on an interval $J$, if
 \begin{align*}
 \left(f(t)-f(s)\right)\left(g(t)-g(s)\right)\geq 0
 \end{align*}
for all $s,t \in J$. It is obvious that, if $f,g$ are monotonic and
have the same monotonicity, then they are synchronic. Dragomir
\cite{DRA2} generalized Chebyshev inequality for convex functions on
a real normed space and applied his results to show that if $p_1,
\cdots, p_n$ is a sequence of nonnegative numbers with
$\sum_{j=1}^np_j=1$ and two sequences $(v_1, \cdots, v_n)$ and
$(u_1, \cdots, u_n)$ in a real inner product space are synchronous,
namely, $\langle v_k-v_j, u_k-u_j\rangle \geq 0$ for all $j,k=1,
\cdots, n$, then $\sum_{k=1}^np_k\langle v_k, u_k\rangle \geq
\langle \sum_{k=1}^np_kv_k, \sum_{k=1}^np_ku_k\rangle$. He also
presented some Chebyshev inequalities for self-adjoint operators
acting on a Hilbert space in \cite{abd}.

In this paper we provide several operator extensions of the
Chebyshev inequality. In the second section, we present our main
results dealing with the Hadamard product of Hilbert space operators
and weighted operator geometric means. The key notion is the
so-called synchronous Hadamard property. More Chebyshev type
inequalities regarding operator means are presented in Section 3. In
Section 4, we apply states on $C^*$-algebras to obtained some
versions related to synchronous functions. We present some Chebyshev
type inequalities involving the singular values of positive $n\times
n$ matrices in the last section.

~~~~~~~~~~~~~~~~~~~~~~~~~~~~~~~~~~~~~~~~~~~~~~~~~~~~~~~~~~~~~~~~~~~~~~~~~~~~~~~~~~~~~~~~~~~~~~~~~~~~~~~~~~~~~~~~~~~~~~~~~~~~~~~~~~~~~~~~~~~~~~~~~~~~~~~~~~~~~~~

\section{Chebyshev inequality dealing with Hadamard product}

This section is devoted to presentation of some operator Chebyshev
inequalities dealing with the Hadamard product. The key notion is the so-called synchronous
Hadamard property.

Let ${\mathscr A}$ be a $C^*$-algebra of operators acting on a Hilbert space and let $T$ be a compact
Hausdorff space equipped with a Radon measure $\mu$. A field $(A_t)_{t\in T}$ of operators in ${\mathscr
A}$ is called a continuous field of operators if the function $t
\mapsto A_t$ is norm continuous on $T$ and the function $t \mapsto
\|A_t\|$ is integrable. Then one can form the Bochner integral
$\int_{T}A_t{\rm d}\mu(t)$, which is the unique element in
${\mathscr A}$ such that
$$\varphi\left(\int_TA_t{\rm d}\mu(t)\right)=\int_T\varphi(A_t){\rm d}\mu(t)$$
for every linear functional $\varphi$ in the norm dual ${\mathscr
A}^*$ of ${\mathscr A}$. By \cite[Page 78]{ped}, since $t\mapsto A_t$ is a continuous function from $T$ to ${\mathscr A}$,
 for every operator $A_t\in{\mathscr A}$ we can consider an element of the form
\begin{align*}
I_\lambda(A_t)=\Sigma_{k=1}^n{A_t(s_k)\mu(E_k)},
\end{align*}
where the $E_k$'$s$ form a partition of $T$ into disjoint Borel subsets, and
\begin{align*}
s_k\in E_k\subseteq \{ t\in T : \|A_t- A_t(s_k)\| \leq \varepsilon\}\hspace{.5cm}(1\leq k\leq n),
\end{align*}
with $\lambda=\{ E_1,\cdots, E_n, \varepsilon \}$. Then
$(I_\lambda(A_t))_{\lambda\in \Lambda}$ is a uniformly convergent
net to $\int _T A_t d\mu(t)$. Let $\mathcal{C}(T,\mathscr A)$ denote
the set of continuous functions on $T$ with values in
${\mathscr A}$. It is easy to see that the set
$\mathcal{C}(T,\mathscr A)$ is a $C^*$-algebra under the pointwise operations and the norm
$\|(A_t)\|=\sup_{t\in T}\|A_t\|$; cf. \cite{han}. Now since tensor product of two operators
is norm continuous, for any operator $B\in{\mathscr A}$ we have
\begin{align*}
\int_{T} (A_t\otimes B)d\mu(t)=\left(\int_{T} A_td\mu(t)\right)\otimes B.
\end{align*}
 Also, for any operator $C\in{\mathscr A}$
\begin{align*}
\int_{T} (C^*A_tC) d\mu(t)=C^*\left(\int_{T} A_t d\mu(t)\right)C.
\end{align*}
Therefore
\begin{align}\label{suit}
&\int_{T} (A_t\circ B)d\mu(t)=\int_{T} V^*(A_t\otimes B)Vd\mu(t)=
V^*\int_{T} (A_t\otimes B)d\mu(t)V\nonumber\\&
=V^*(\int_{T} A_td\mu(t)\otimes
B)V =\int_{T} A_td\mu(t)\circ B\qquad (A_t, B\in{\mathscr A}).
\end{align}
Let us give our key definition.

\begin{definition}
Two fields $(A_t)_{t\in T}$ and $(B_t)_{t\in T}$
 are said to have the synchronous Hadamard property if
\begin{align*}
\left(A_t-A_s\right)\circ\left(B_t-B_s\right)\geq0
\end{align*}
for all $s, t\in T$.
\end{definition}
\noindent The first result reads as follows.
\begin{theorem}\label{20}
 Let ${\mathscr A}$ be a $C^*$-algebra, $T$ be a compact
 Hausdorff space equipped with a Radon measure $\mu$, let $(A_t)_{t\in T}$ and $(B_t)_{t\in T}$ be fields in $\mathcal{C}(T,\mathscr A)$ with the synchronous Hadamard property and let $\alpha: T\rightarrow [0, +\infty)$ be a measurable function. Then
\begin{align}\label{22}
\int_{T} \alpha(s) d\mu(s)\int_{T}\alpha(t)(A_t\circ B_t) d\mu(t)\geq\left(\int_{T}\alpha(t) A_t d\mu(t)\right)\circ\left(\int_{T}\alpha(s) B_s d\mu(s)\right).
 \end{align}
 \end{theorem}
\begin{proof}
 We have
 \begin{align*}
 &\int_{T} \alpha(s) d\mu(s)\int_{T} \alpha(t)(A_t\circ B_t) d\mu(t)-\left(\int_{T}\alpha(t) A_t d\mu(t)\right)\circ\left(\int_{T} \alpha(s)B_s d\mu(s)\right)\\&=
 \int_{T}\int_{T} \alpha(s)\alpha(t)(A_t\circ B_t) d\mu(t)d\mu(s)-\int_{T}\left(\int_{T}\alpha(t) A_t d\mu(t)\right)\circ \alpha(s)B_s d\mu(s)\\&
 \qquad\qquad \qquad\qquad \qquad\qquad\qquad\qquad \qquad\qquad\qquad \qquad \qquad (\textrm {by~} \ref{suit})\\&
 =\int_{T} \int_{T} \alpha(s)\alpha(t)(A_t\circ B_t) d\mu(t)d\mu(s)-\int_{T}\int_{T}\alpha(t)\alpha(s)( A_t \circ B_s) d\mu(t)d\mu(s)\\&
 \qquad\qquad \qquad\qquad \qquad\qquad\qquad\qquad \qquad\qquad\qquad \qquad \qquad (\textrm {by~} \ref{suit})\\&
 =\int_{T}\int_{T}\left(\alpha(s)\alpha(t)(A_t\circ B_t)-\alpha(t)\alpha(s)(A_t\circ B_s)\right) d\mu(t) d\mu(s)\\&=
 {1\over 2}\int_{T}\int_{T}\Big{[}\alpha(s)\alpha(t)(A_t\circ B_t)-\alpha(t)\alpha(s)(A_t\circ B_s)\\&
 \quad-\alpha(s)\alpha(t)(A_s\circ B_t)+\alpha(t)\alpha(s)(A_s\circ B_s)\Big{]} d\mu(t) d\mu(s)\\&=
 {1\over2}\int_{T}\int_{T}\Big{[}\alpha(s)\alpha(t)\left(A_t-A_s\right)\circ\left(B_t-B_s\right)\Big{]} d\mu(t) d\mu(s)
 \\&\geq0.\qquad(\textrm{since the fields}\hspace{.1cm}(A_t)\hspace{.1cm} \textrm{and}\hspace{.1cm} (B_t)\hspace{.1cm}
 \textrm{have synchronous Hadamard property})
 \end{align*}
 \end{proof}

 In the discrete case $T=\{1,\cdots, n\}$, set $\alpha(i)=\omega_i\geq0$ $(1\leq i\leq n)$.
 Then Theorem \ref{20} forces the following result.
\begin{corollary}\cite[Teorem 2.1]{math}\label{31}
 Let $A_1\geq \cdots \geq A_n$, $B_1\geq \cdots \geq B_n$ be self-adjoint operators and
 $\omega_1, \cdots, \omega_n$ be positive numbers. Then
 \begin{align*}
 \sum_{ j=1}^n \omega_j\sum_{ j=1}^n\omega_j (A_j\circ B_j)\geq\left(\sum_{ j=1}^n \omega_j A_j\right)\circ\left(\sum_{ j=1}^n
 \omega_j B_j\right).
\end{align*}
\end{corollary}

\section{More Chebyshev type inequalities regarding operator means}

Recall that a continuous function $f:J\to \mathbb{R}$ is
supers-multiplicative if $f(xy)\geq f(x)f(y)$, for all $x, y\in J$.
In the next result we need the notion of increasing field. Let $T$
be a compact Hausdorff space as well as a totally ordered set under an
order $\preceq$. We say $(A_t)$ is an increasing (decreasing, resp.)
field, whenever $t\preceq s$ implies that $A_t\leq A_s$ ($A_t\geq
A_s$, resp.). In this section we frequently employ some known relationships between the Hadamard product and operator means; cf. \cite[Chapter 6]{abc}.

\begin{theorem}\label{m32}
Let ${\mathscr A}$ be a $C^*$-algebra, $T$ be a compact Hausdorff
space equipped with a Radon measure $\mu$ being also a totally ordered set, let
$(A_t)_{t\in T}, (B_t)_{t\in T}, (C_t)_{t\in T}, (D_t)_{t\in T} $ be
positive increasing fields in $\mathcal{C}(T,\mathscr A)$, let
$\alpha: T\rightarrow [0, +\infty)$ be a measurable function and
$\sigma$ be an operator mean with the super-multiplicative representing
function. Then
\begin{align*}
\int_{T} \alpha(s) d\mu(s)&\int_{T}\alpha(t)\left((A_t\circ B_t) \sigma(C_t\circ D_t)\right) d\mu(t)\\&\geq\left(\int_{T}\alpha(t) (A_t\sigma C_t) d\mu(t)\right)\circ\left(\int_{T}\alpha(s) (B_s\sigma D_s) d\mu(s)\right).
\end{align*}
\end{theorem}

\begin{proof}
 Let $s, t\in T$. Without loss of generality assume that $s\preceq t$. Then by the property $(\textrm{i})$ of the operator mean we have $0\leq(A_t\sigma B_t)-(A_s\sigma B_s)$. Then
\begin{align*}
 &\int_{T} \alpha(s) d\mu(s)\int_{T}\left(\alpha(t)(A_t\circ B_t)\sigma(C_t\circ D_t)\right) d\mu(t)\\&
 \hspace{.3cm}-\left(\int_{T}\alpha(t) (A_t\sigma C_t) d\mu(t)\right)\circ\left(\int_{T}\alpha(s) (B_s\sigma D_s) d\mu(s)\right)\\&=\int_{T}\int_{T}\alpha(s)\alpha(t)\left((A_t\circ B_t)\sigma(C_t\circ D_t)\right) d\mu(t) d\mu(s)\\&
 \hspace{.3cm}-\int_{T}\int_{T}\alpha(t)\alpha(s) \left((A_t\sigma C_t)\circ (B_s\sigma D_s)\right) d\mu(t) d\mu(s)\qquad (\textrm{by~} \ref{suit})\\&
 \geq\int_{T} \int_{T}\alpha(s)\alpha(t)\left((A_t\sigma C_t)\circ (B_t\sigma D_t)\right) d\mu(t) d\mu(s)
\\&\hspace{.3cm}-\int_{T}\int_{T}\alpha(t)\alpha(s) \left((A_t\sigma C_t)\circ (B_s\sigma D_s)\right) d\mu(t) d\mu(s)\qquad (\textrm{by \cite[Theorem 6.7]{abc}}) \\&={1\over2}\int_{T}\int_{T}\alpha(s)\alpha(t)\Big{[}\left((A_t\sigma C_t)\circ (B_t\sigma D_t)\right)- \left((A_t\sigma C_t)\circ (B_s\sigma D_s)\right)\\&\hspace{.3cm}+ \left((A_s\sigma C_s)\circ (B_s\sigma D_s)\right)- \left((A_s\sigma C_s)\circ (B_t\sigma D_t)\right)\Big{]} d\mu(t) d\mu(s)\\&
={1\over2}\int_T\int_{T}\alpha(s)\alpha(t)\Big{[}(A_t\sigma C_t)- (A_s\sigma C_s)\Big{]}\circ\Big{[}(B_t\sigma D_t)- (B_s\sigma D_s)\Big{]} d\mu(t) d\mu(s)\\&\geq0. \qquad\qquad \qquad\qquad\qquad \qquad \qquad \qquad\qquad \qquad \qquad (\textrm{by~} \eqref{nn2})
\end{align*}
\end{proof}
A discrete version of the theorem above is the following result
obtained by taking $T=\{1,\cdots, n\}$.
 \begin{corollary}
Let $A_{i+1}\geq A_{i}\geq0$, $B_{i+1}\geq B_{i}\geq0$, $C_{i+1}\geq
C_{i}\geq0$,
 $D_{i+1}\geq D_{i}\geq0\,\,(1\leq i\leq n-1)$, $\omega_1, \cdots, \omega_n$
 be positive numbers and $\sigma$ be an operator mean with the super-multiplicative representing function. Then
\begin{align*}
\sum_{ j=1}^n \omega_j\sum_{ j=1}^n\omega_j \Big{[}(A_j\circ B_j)\sigma(C_j\circ D_j)\Big{]}
\geq\left(\sum_{ j=1}^n \omega_j (A_j\sigma C_j)\right)\circ\left(\sum_{ j=1}^n \omega_j (B_j\sigma D_j)\right).
\end{align*}
 \end{corollary}
 \begin{theorem}\label{41}
Let ${\mathscr A}$ be a $C^*$-algebra, $T$ be a compact Hausdorff
space equipped with a Radon measure $\mu$ being also a totally ordered set, let
$(A_t)_{t\in T}, (B_t)_{t\in T} $ be positive increasing fields in
$\mathcal{C}(T,\mathscr A)$ and let $\alpha: T\rightarrow [0,
+\infty)$ be a measurable function. Then
\begin{align*}
\int_{T} \alpha(s) d\mu(s)\int_{T}\alpha(t)(A_t\circ B_t) d\mu(t)\geq\left(\int_{T}\alpha(t) (A_t\sharp_\mu B_t) d\mu(t)\right)\circ\left(\int_{T}\alpha(s) (A_s\sharp_{1-\mu} B_s) d\mu(s)\right)
\end{align*}
for all $\mu\in[0,1]$.
\end{theorem}

\begin{proof}
 Let $s, t\in T$. Without loss of generality assume that $s\preceq t$. Then by the property $(\textrm{i})$ of the operator mean, we have $0\leq(A_t\sharp_\mu B_t)-(A_s\sharp_\mu B_s)$ and $0\leq(A_t\sharp_{1-\mu} B_t)-(A_s\sharp_{1-\mu} B_s)$. Then
\begin{align*}
 &\int_{T} \alpha(s) d\mu(s)\int_{T}\alpha(t)(A_t\circ B_t) d\mu(t)-\left(\int_{T}\alpha(t) (A_t\sharp_\mu B_t) d\mu(t)\right)\circ\left(\int_{T}\alpha(s) (A_s\sharp_{1-\mu} B_s) d\mu(s)\right)\\&
=\int_{T} \int_{T}\alpha(s)\alpha(t)(A_t\circ B_t)
d\mu(t)d\mu(s)-\int_{T}\int_{T}\alpha(t)\alpha(s)\left((A_t\sharp_\mu
B_t) \circ (A_s\sharp_{1-\mu} B_s)\right) d\mu(t)d\mu(s)\\&
\qquad\qquad\qquad\qquad\qquad\qquad\qquad\qquad\qquad
\qquad\qquad\qquad \qquad \qquad (\textrm{by~} \ref{suit})\\&\geq
\int_{T} \int_{T}\alpha(s)\alpha(t)\left((A_t\sharp_\mu B_t)\circ
(A_t\sharp_{1-\mu}B_t)\right)
d\mu(t)d\mu(s)\\&\hspace{.3cm}-\int_{T}\int_{T}\alpha(t)\alpha(s)\left((A_t\sharp_\mu
B_t) \circ (A_s\sharp_{1-\mu} B_s)\right) d\mu(t)d\mu(s)\qquad
(\textrm{by \cite[Theorem 6.6]{abc}})\\&={1\over2}\int_{T}
\int_{T}\Big{[}\alpha(s)\alpha(t)\left((A_t\sharp_\mu B_t)\circ
(A_t\sharp_{1-\mu}B_t)\right)
-\alpha(t)\alpha(s)\left((A_t\sharp_\mu B_t) \circ
(A_s\sharp_{1-\mu} B_s)\right)\\&
\hspace{.3cm}+\alpha(t)\alpha(s)\left((A_s\sharp_\mu B_s)\circ
(A_s\sharp_{1-\mu}B_s)\right)
-\alpha(s)\alpha(t)\left((A_s\sharp_\mu B_s) \circ
(A_t\sharp_{1-\mu} B_t)\right)\Big{]} d\mu(t)d\mu(s)\\&
={1\over2}\int_{T}\int_{T}\alpha(s)\alpha(t)\Big{[}(A_t\sharp_{\mu}
B_t)- (A_s\sharp_{\mu} B_s)\Big{]}\circ\Big{[}(A_t\sharp_{1-\mu}
B_t)- (A_s\sharp_{1-\mu} B_s)\Big{]} d\mu(t) d\mu(s)\\&\geq0.
\qquad\qquad\qquad\qquad\qquad \qquad\qquad\qquad \qquad\qquad\qquad\qquad \qquad (\textrm{by~} \eqref{nn2})
\end{align*}
\end{proof}
 In the discrete case $T=\{1,\cdots, n\}$, setting $\alpha(i)=\omega_i \geq0\,\,(1\leq i\leq n)$ in Theorem \ref{41} we reach the next assertion.
\begin{corollary}
Let $A_n\geq \cdots \geq A_1\geq0$, $B_n\geq \cdots \geq B_1\geq0$
and $\omega_1, \cdots, \omega_n$ be positive numbers. Then
\begin{align*}
\sum_{ j=1}^n \omega_j\sum_{ j=1}^n\omega_j \left(A_j\circ B_j\right)\geq\left(\sum_{ j=1}^n \omega_j (A_j\sharp_\mu B_j)\right)\circ\left(\sum_{ j=1}^n \omega_j (A_j\sharp_{1-\mu}B_j)\right)
\end{align*}
for all $\mu \in[0,1]$.
 \end{corollary}

\begin{proposition}
Let $f:[0,\infty)\to \mathbb{R}$ be a super-multiplicative
and operator monotone function, $A_1\geq \cdots \geq A_n\geq0$, $B_1\geq
\cdots \geq B_n\geq0$ and $\omega_1, \cdots, \omega_n$ be positive
numbers. Then
\begin{align*}
\sum_{ j=1}^n \omega_j\sum_{ j=1}^n\omega_j f(A_j\circ B_j)\geq\left(\sum_{ j=1}^n \omega_j f(A_j)\right)\circ\left(\sum_{ j=1}^n \omega_j f(B_j)\right).
\end{align*}
 \end{proposition}

\begin{proof}
\begin{align*}
&\sum_{ j=1}^n \omega_j\sum_{ j=1}^n\omega_j f(A_j\circ
B_j)-\left(\sum_{ j=1}^n \omega_j f(A_j)\right)\circ\left( \sum_{
j=1}^n \omega_j f(B_j)\right)\\& \geq\sum_{ j=1}^n \omega_j\sum_{
j=1}^n\omega_j \left(f(A_j)\circ f(B_j)\right) -\left(\sum_{
j=1}^n \omega_j f(A_j)\right)\circ\left(\sum_{ j=1}^n \omega_j
f(B_j)\right)\\
&\qquad\qquad\qquad\qquad\qquad\qquad\qquad\qquad\qquad\qquad\qquad\qquad\qquad\qquad(\textrm{by \cite[Theorem
6.3]{abc}})\\&=\sum_{i, j=1}^n \Big{[}\omega_i\omega_j
\left(f(A_j)\circ f(B_j)\right) - \omega_i\omega_j
\left(f(A_i)\circ f(B_j)\right)\Big{]}\\& ={1\over2}\sum_{i, j=1}^n
\omega_i\omega_j \Big{[}\left(f(A_j)\circ f(B_j)\right) -
\left(f(A_i)\circ f(B_j)\right)+ \left(f(A_i)\circ f(B_i)\right)-
\left(f(A_j)\circ f(B_i)\right)\Big{]}\\& ={1\over 2}\sum_{i,
j=1}^n\omega_i\omega_j\Big{[}\left(f(A_j)-f(A_i)\right)\circ
\left(f(B_j)-f(B_i)\right)\Big{]} \geq0 \\
& \qquad\qquad\qquad\qquad \qquad\qquad\qquad\qquad\qquad\qquad\qquad(\textrm{by
the operator monotonicity of~} f)\,.
\end{align*}
\end{proof}

\begin{example}
Let $A_1\geq \cdots \geq A_n\geq0$, $B_1\geq \cdots \geq B_n\geq0$
and $\omega_1, \cdots, \omega_n$ be positive numbers. Then
\begin{align*}
\sum_{ j=1}^n \omega_j\sum_{ j=1}^n\omega_j (A_j\circ B_j)^p\geq\left(\sum_{ j=1}^n \omega_j A_j^p\right)\circ\left(\sum_{ j=1}^n \omega_j B_j^p\right)
\end{align*}
for each $p\in[0,1]$.
\end{example}

In the finite dimensional case we get the following.
\begin{proposition}
 Let $A_1\geq \cdots \geq A_k\geq 0$, $B_1\geq \cdots \geq B_k\geq 0$ be $n\times n$ matrices and
 $\omega_1, \cdots, \omega_k$ be positive numbers. Then
 \begin{align*}
 \left(\sum_{ j=1}^k \omega_j\right)^n{\rm det }\left(\sum_{ j=1}^k\omega_j (A_j\circ B_j)\right)\geq\left(\sum_{ j=1}^k \omega_j^n {\rm det }(A_j)\right)\left(\sum_{ j=1}^k \omega_j^n {\rm det }(B_j)\right).
\end{align*}
\end{proposition}
\begin{proof}
\begin{align*}
 &\left(\sum_{ j=1}^k \omega_j\right)^n{\rm det }\left(\sum_{ j=1}^k\omega_j (A_j\circ B_j)\right)=
 {\rm det }\left(\sum_{ j=1}^k \omega_j\sum_{ j=1}^k\omega_j (A_j\circ B_j)\right)\\&\geq
 {\rm det }\left(\left(\sum_{ j=1}^k \omega_j A_j\right)\circ\left(\sum_{ j=1}^k \omega_j B_j\right)\right)\qquad\quad (\textrm{by Corolary} \hspace{.1cm}\ref{31})\\&\geq
 {\rm det }\left(\sum_{ j=1}^k \omega_j A_j\right){\rm det }\left(\sum_{ j=1}^k \omega_j B_j\right)\qquad\qquad (\textrm{by~ \cite[Theorem 7.27]{Zhang}})\\&\geq
 \left(\sum_{ j=1}^k \omega_j^n {\rm det }(A_j)\right)\left(\sum_{ j=1}^k \omega_j^n {\rm det }(B_j)\right) \qquad\quad (\textrm{by~ \cite[Theorem 7.7]{Zhang}}).
\end{align*}
\end{proof}
 \begin{proposition}
Let $A_1\geq \cdots \geq A_k> 0$, $B_k\geq \cdots \geq B_1\geq 0$ be
$n\times n$ matrices and
 $\omega_1, \cdots, \omega_k$ be positive numbers. Then
 \begin{align*}
 \left(\sum_{ j=1}^k \omega_j\right)\left(\sum_{ j=1}^k\omega_j{\rm tr }(A_j^{-1}B_j) \right)\geq\left(\sum_{ j=1}^k \omega_j {\rm tr }(A_j)^{-1}\right)\left(\sum_{ j=1}^k \omega_j {\rm tr }(B_j)\right).
\end{align*}
\begin{proof}
\begin{align*}
 &\left(\sum_{ j=1}^k \omega_j\right)\left(\sum_{ j=1}^k\omega_j{\rm tr }(A_j^{-1}B_j) \right)\\&\geq
 \left(\sum_{ j=1}^k \omega_j\right)\left(\sum_{ j=1}^k\omega_j{\rm tr }(A_j)^{-1}{\rm tr }(B_j) \right)\qquad(\textrm{by \cite[page 224]{Zhang})}\\&\geq
 \left(\sum_{ j=1}^k \omega_j {\rm tr }(A_j)^{-1}\right)\left(\sum_{ j=1}^k \omega_j {\rm tr }(B_j)\right). \qquad(\textrm {by Chebyshev inequality})
\end{align*}
\end{proof}
\end{proposition}


\section{Chebyshev inequality for synchronous functions involving states }

In this section, we apply the continuous functional calculus to
synchronous functions and present some Chebyshev type inequalities
involving states on $C^*$-algebras. Our main result of this section
reads as follows.

\begin{theorem}\label{30}
 Let $\mathscr A$ be a unital $C^*$-algebra, $\tau_1,\tau_2$ be states on $\mathscr A$ and $f,g:J\to \mathbb {R}$ be synchronous functions. Then
 \begin{eqnarray}\label{m6}
 \tau_1\left(f(A)g(A)\right)+\tau_2\left(f(B)g(B)\right)
 \geq\tau_1\left(f(A)\right)\tau_2\left(g(B)\right)+\tau_2\left(f(B)\right)\tau_1\left(g(A)\right)
 \end{eqnarray}
 for all $A, B\in{\mathbb B}_h^J({\mathscr H})$.
 \end{theorem}

\begin{proof}
 For the synchronous functions $f,g$ and for each $s,t\in J$
 \begin{align*}\
 f(t)g(t)+f(s)g(s)-f(t)g(s)-f(s)g(t)\geq0.
 \end{align*}
 Fix $s\in J$. By the functional calculus for the operator $A$ we have
 \begin{align*}\
 f(A)g(A)+f(s)g(s)-f(A)g(s)-f(s)g(A)\geq0,
 \end{align*}
 whence
 \begin{align*}\
 \tau_1\left(f(A)g(A)\right)+f(s)g(s)-\tau_1\left(f(A)\right)g(s)-f(s)\tau_1\left(g(A)\right)\geq0.
 \end{align*}
 Now for the operator $B$
 \begin{align*}\
 \tau_1\left(f(A)g(A)\right)+f(B)g(B)-\tau_1\left(f(A)\right)g(B)-f(B)\tau_1\left(g(A)\right)\geq0.
 \end{align*}
 For the state $\tau_2$ we have
 \begin{align*}\
\tau_1\left(f(A)g(A)\right)+\tau_2\left(f(B)g(B)\right)
 \geq\tau_1\left(f(A)\right)\tau_2\left(g(B)\right)+\tau_2\left(f(B)\right)\tau_1\left(g(A)\right).
 \end{align*}
 \end{proof}
 \begin{example}
(i) Let $\tau$ be a state on ${\mathbb B}({\mathscr H})$ and
$p,q>0$. Since $f(t)=t^p$ and $g(t)=t^q$ are synchronous
 \begin{align*}\
 \tau(A^{p+q})+\tau(B^{p+q})\geq\tau(A^{p})\tau(B^{q})+\tau(B^p)\tau(A^q)\qquad(A, B\geq0).
 \end{align*}
 In a similar fashion, for self-adjoint operators $A, B\in{\mathbb B}({\mathscr H})$
 \begin{align*}\
 \tau(e^{\alpha A+\beta A})+ \tau(e^{\alpha B+\beta B})\geq\tau(e^{\alpha A})\tau(e^{\beta B})+\tau(e^{\beta B})\tau(e^{\alpha A})\qquad(\alpha, \beta \geq0).
 \end{align*}
(ii) Let $A, B$ be positive matrices, $C$ be a positive definite
matrix with ${\rm tr}(C)=\alpha$ and $p, q\geq0$. Utilizing
$\tau(A)=\frac{1}{\alpha}{\rm tr}(A\circ C)$ in (i) we have
 \begin{align*}
 {\rm tr}(A^{p+q}\circ C+ B^{p+q}\circ C)\geq{1\over \alpha} \left({\rm tr} (A^p\circ C){\rm tr}(B^q\circ C)+{\rm tr}(A^q\circ C){\rm tr}(B^p\circ C)\right).
 \end{align*}
 (iii) Let $f,g:J\rightarrow \mathbb {R}$ be synchronous functions.
Then for $n\times n$ matrices $A,B$ with spectra in $J$
 \begin{align*}
 {\rm tr}\left(f(A)g(A)+f(B)g(B)\right)\geq{1\over n} \left({\rm tr} \left(f(A)\right){\rm tr}\left(g(B)\right)+{\rm tr}\left(g(A)\right){\rm tr}\left(f(B)\right)\right).
 \end{align*}
 \end{example}
 Using Theorem \ref{30} we obtain two next corollaries.
\begin{corollary}
 Let $\mathscr A$ be a unital $C^*$-algebra, $\tau$ be a state on $\mathscr A$ and $f,g:J\rightarrow \mathbb {R}$ be synchronous functions. Then
 \begin{align*}\
 \tau\left(f(A)g(A)\right)\geq\tau\left(f(A)\right)\tau\left(g(A)\right)
 \end{align*}
 for all operator $A\in{\mathbb B}_h^J({\mathscr H})$.
 \end{corollary}
\begin{proof}
 Put $B=A$ in inequality (\ref{m6}) to get the result.
 \end{proof}

\begin{corollary}\cite[Theorem 1]{abd}
 Let $f,g:J\rightarrow \mathbb {R}$ be synchronous functions. Then
 \begin{align*}\
 \langle f(A)g(A)x,x\rangle+\langle f(B)f(B)y,y\rangle\geq\langle f(A)x,x\rangle \langle g(B)y,y\rangle+\langle f(B)y,y\rangle \langle g(A)x,x\rangle
 \end{align*}
 for all operators $A, B\in{\mathbb B}_h^J({\mathscr H})$ and all unit vectors $x, y \in {\mathscr H}$.
 \end{corollary}
\begin{proof}
 Apply Theorem \ref{30} to the states $\tau_1, \tau_2$ defined by $\tau_1(A)=\langle Ax,x\rangle, \tau_2(A)=\langle Ay,y\rangle$ $\hspace{.2cm}(A\in {\mathbb B}({\mathscr H}))$ for fixed unit vectors $x, y\in {\mathscr H}$.
 \end{proof}

Using the same strategy as in the proof of \cite[Lemma 2.1]{aczel}
we get the next theorem.
\begin{theorem}\label{23}
 Let $\tau$ be a state on a unital $C^*$-algebra $\mathscr A$ and $f:J\rightarrow [0,+\infty)$, $g:J\rightarrow \mathbb {R}$ be continuous functions such that $f$ is decreasing and $g$ is operator decreasing on a compact interval $J$. Then
 \begin{align*}
 \tau\left(f(A)g(A)\right)\geq\tau\left(f(B)\right)\tau\left(g(A)\right)
 \end{align*}
 for all $A, B\in{\mathbb B}_h^J({\mathscr H})$ with $A\leq B$.
 \end{theorem}

\begin{proof}
 Put $\alpha=\inf_{x\in J} g(x)$ and $\beta=\sup_{x\in J}g(x)$. Then $\alpha\leq g(x)\leq \beta \hspace{.2cm}(x\in J)$. So $ \alpha I\geq g(B)\geq \beta I $, whence $ \alpha \geq \tau\left(g(B)\right)\geq \beta$. Therefore, there exists a number $t_0\in J$ such that $\tau(g(B))=g(t_0)$.\\
Then 
if $x\in J, x\geq t_0$, then $g(x)\leq \tau\left(g(B)\right), f(x)\leq f(t_0)$, and if $x\in J, x\leq t_0$, then $g(x)\geq \tau\left(g(B)\right), f(x)\geq f(t_0)$. Hence
 \begin{align*}
 \left(f(x)-f(t_0)\right)\left(g(x)-\tau\left(g(B)\right)\right)\geq0
 \end{align*}
 for all $x\in J$. Thus
 \begin{align*}
 f(x)\left(g(x)-\tau\left(g(B)\right)\right)\geq f(t_0)\left(g(x)-\tau\left(g(B)\right)\right)
 \end{align*}
 for all $x\in J$. Hence
 \begin{align*}
 f(A)\left(g(A)-\tau\left(g(B)\right)\right)\geq f(t_0)\left(g(A)-\tau\left(g(B)\right)\right).
 \end{align*}
 Now
 \begin{align*}
 \tau\left(f(A)g(A)\right)-\tau \left(g(B)\right)\tau \left(f(A)\right)&=\tau \left(f(A)\left(g(A)-\tau(g(B))\right)\right)\\&\geq \tau \left(f(t_0)\left(g(A)-\tau(g(B))\right)\right)\\&=f(t_0)
 \left(\tau \left(g(A)\right)-\tau\left(g(B)\right)\right) \\&\geq0.\hspace{.2cm} (\textrm{since g is operator decreasing})
 \end{align*}
 \end{proof}

\begin{remark}
The assumption $A\leq B$ is necessary in Theorem \ref{23}, since if $\tau(A)={1\over2}{\rm tr}(A)$ on $\mathbb{M}_2$, $f(t)=g(t)={1\over t}$, $A=\left(\begin{array}{cc}
 2&0\\
 0&3
 \end{array}\right)$
 and $B=\left(\begin{array}{cc}
 1&0\\
 0&1
 \end{array}\right)$,
 then we observe that $A\nleq B$ and $\tau(A^{-2})={13\over 72}<{5\over 12}=\tau(A^{-1})\tau(B^{-1})$.
 \end{remark}

\begin{corollary}
 Suppose that $f:J\rightarrow [0,+\infty)$ and $g:J\rightarrow \mathbb {R}$ are continuous functions such that $f$ is decreasing and $g$ is operator decreasing. Then
 \begin{align*}
 \langle f(A)g(A)x,x\rangle-\langle f(B)x,x\rangle \langle g(A)x,x\rangle\geq0
 \end{align*}
 for all operators $A, B\in{\mathbb B}_h^J({\mathscr H})$ such that $A\leq B$ and all unit vector $x\in {\mathscr H}$.
 \end{corollary}
\begin{proof}
 Apply Theorem \ref{23} to the state $\tau$ defined by $\tau(A)=\langle Ax,x\rangle$$\hspace{.1cm}(A\in{\mathbb B}({\mathscr H}))$ for a fixed unit vector $x\in H$.
 \end{proof}
Using the same strategy as in the proof of Theorem \ref{30} we get
the next result.
\begin{theorem}\label{40}
 Let $\mathscr A$ be a unital $C^*$-algebra, $\tau_1, \tau_2$ be states on $\mathscr A$ and $f,g:\mathbb {R}\rightarrow \mathbb {R}$ be synchronous functions. Then
 \begin{eqnarray}\label{m4}
 \tau_2\left(f(A)g(A)\right)+f\left(\tau_1 (B)\right)g\left(\tau_1 (B)\right)\geq f\left(\tau_1 (A)\right)\tau_2\left(g(B)\right)+\tau_1\left(f(B)\right)g\left(\tau_2(A)\right)
 \end{eqnarray}
 for all self-adjoint operators $A, B$.
 \end{theorem}
We now get immediately the next corollaries.
\begin{corollary}
 Let $f,g:J\rightarrow \mathbb {R}$ be synchronous functions. Then
 \begin{align*}\
 \langle f(A)g(A)x,x\rangle+f(\langle By,y\rangle)g(\langle By,y\rangle)\geq f(\langle Ax,x\rangle) \langle g(B)y,y\rangle+\langle f(B)y,y\rangle g(\langle Ax,x\rangle)
 \end{align*}
 for all operators $A, B\in{\mathbb B}_h^J({\mathscr H})$ and all unit vectors $x, y \in {\mathscr H}$.
 \end{corollary}
\begin{corollary}\cite[Theorem 2]{abd}
 Let $f,g:J\rightarrow \mathbb {R}$ are synchronous functions. Then
 \begin{align*}\
& \hspace{-2cm}\langle f(A)g(A)x,x\rangle-f(\langle Ax,x\rangle)g(\langle Ax,x\rangle)\\
 &\geq [\langle f(A)x,x\rangle-f(\langle Ax,x\rangle)][ g(\langle Ax,x\rangle)-\langle g(A)x,x\rangle]
 \end{align*}
 for all operator $A\in{\mathbb B}_h^J({\mathscr H})$ and any unit vector $x\in {\mathscr H}$.
 \end{corollary}

 \begin{corollary}\label{38}
 Let $\mathscr A$ be a unital $C^*$-algebra, $\tau$ be a state on $\mathscr A$ and $f,g:\mathbb {R}\rightarrow \mathbb {R}$ be synchronous functions. Then
 \begin{align*}
 \tau\left(f(B)g(B)\right)-\tau\left(f(B)\right)\tau\left(g(B)\right)\geq\left(\tau\left(f(B)\right)-f\left(\tau(A)\right)\right)\left(g\left(\tau(A)\right)
 -\tau\left(g(B)\right)\right)
 \end{align*}
 for all self-adjoint operators $A, B$.
 \end{corollary}

\begin{proof}
 By using inequality (\ref{m4}) we get
 \begin{align*}\
 & \tau\left(f(B)g(B)\right)-\tau\left(f(B)\right)\tau\left(g(B)\right)
 \\&\geq f\left(\tau (A)\right)\tau\left(g(B)\right)+\tau\left(f(B)\right)g\left(\tau(A)\right)-f\left(\tau (A)\right)g\left(\tau (A)\right)-\tau\left(f(B)\right)\tau\left(g(B)\right)
 \\&=\left(\tau\left(f(B)\right)-f\left(\tau(A)\right)\right)\left(g\left(\tau(A)\right)-\tau\left(g(B)\right)\right).
 \end{align*}
 \end{proof}
 By using Corollary \ref{38} and the Davis--Choi--Jensen inequality \cite{abc} we obtain the next result.
\begin{corollary}
 Let $\mathscr A$ be a unital $C^*$-algebra, $\tau$ be a state on $\mathscr A$ and $f,g:\mathbb {R}\rightarrow \mathbb {R}$ be synchronous such that one of them is convex while the other is concave on $\mathbb {R}$. Then
 \begin{align*}
 \tau\left(f(A)g(A)\right)-\tau\left(f(A)\right)\tau\left(g(A)\right)\geq\left(\tau\left(f(A)\right)-f\left(\tau(A)\right)
 \right)\left(g\left(\tau(A)\right)-\tau\left(g(A)\right)\right)\geq0
 \end{align*}
 for all self-adjoint operator $A$.
 \end{corollary}
 In the next proposition we establish a version of Acz\'{e}l--Chebyshev type inequality.
\begin{proposition}
 Let $\mathscr A$ be a unital $C^*$-algebra, $\tau$ be a state on $\mathscr A$ and $f, g$ be continuous functions such that $0\leq f(x)\leq \alpha$ and $0\leq g(x)\leq \beta$ for some non-negative real numbers $\alpha, \beta$. Then
 \begin{align}\label{m20}
 \left(\alpha\beta-\tau\left(f(B)g(B)\right)\right)\geq\left(\alpha-\tau\left(f(B)\right)\right)\left(\beta-\tau\left(g(A)\right)\right)
 \end{align}
 for all positive operators $A, B\in\mathscr A$.
 \end{proposition}

\begin{proof}
 If $\alpha=0$ or $\beta=0$, inequality (\ref{m20}) is travail. Now assume that $\alpha>0$ and $\beta>0$. Then (\ref{m20}) is equivalent to the inequality
 \begin{align*}
 \left(1-\tau\left(f(B)g(B)\right)\right)\geq\left(1-\tau\left(f(B)\right)\right)\left(1-\tau\left(g(A)\right)\right),
 \end{align*}
 with $0\leq f(x)\leq 1$ and $0\leq g(x)\leq 1$. Then we have
 \begin{align*}
 \left(1-\tau\left(f(B)g(B)\right)\right)\geq\left(1-\tau\left(f(B)\right)\right)\geq \left(1-\tau\left(f(B)\right)\right)\left(1-\tau\left(g(A)\right)\right)\geq0.
 \end{align*}
 \end{proof}

\section{Chebyshev type inequalities involving singular values }

In this section we deal with some singular value versions of the
Chebyshev inequality for positive $n\times n$ matrices. We need the
following known result.

\begin{lemma}\label{m13}\cite[Corollary III.2.2]{bha}
 Let $A, B$ be $n\times n$ Hermitian matrices. Then
 \begin{align*}
 \lambda_j^\downarrow(A+B) \geq\lambda_n^\downarrow(A)+\lambda_j^\downarrow(B)\hspace{.5cm} (1\leq j\leq n).
 \end{align*}
 \end{lemma}

\begin{theorem}\label{36}
Let $f,g:[0,+\infty)\rightarrow {[0,+\infty)}$ be synchronous
functions. Then
\begin{align*}
s_j\left(f(A)g(A)\right)+s_j\left(f(B)g(B)\right)&\geq
s_n\left(f(A)\right)s_n\left(g(B)\right)\nonumber\\&\quad+{1\over2}\left(
s_j\left(g(A)\right)s_j\left(f(B)\right)+s_j\left(g(B)\right)s_j\left(f(A)\right)\right)
\end{align*}
for all positive matrices $A,B\in\mathbb{M}_n$ and all $j=1, 2,
\cdots, n$.
\end{theorem}

\begin{proof}
For synchronous functions $f,g$ we have
\begin{align*}\
 f(t)g(t)+f(s)g(s)\geq f(t)g(s)+f(s)g(t)\hspace{.3cm}( s, t\geq 0).
\end{align*}
 If we fix $s\in[0,+\infty)$, then
\begin{align*}\
 f(A)g(A)+f(s)g(s)I\geq f(A)g(s)+f(s)g(A).
\end{align*}
Hence
\begin{align*}\
 s_j\left(f(A)g(A)\right)+f(s)g(s)&= s_j\left(f(A)g(A)+f(s)g(s)\right)\\&\geq s_j\left(f(A)g(s)+f(s)g(A)\right)\\&\geq s_n\left(f(A)g(s)\right)+s_j\left(f(s)g(A)\right)\\
&\qquad\qquad \qquad\qquad\qquad \qquad (\textrm{by inequality (\ref{m13}) })\\&=s_n\left(f(A)\right)g(s)+f(s)s_j\left(g(A)\right)\quad(1\leq j\leq n).
\end{align*}
 Using functional calculus for $B$ we get
\begin{align*}\
& s_j\left(f(A)g(A)\right)+f(B)g(B)\geq
s_n\left(f(A)\right)g(B)+s_j\left(g(A)\right)f(B)\quad(1\leq j\leq
n).
\end{align*}
Thus
\begin{align}\label{m10}
 s_j\left(f(A)g(A)\right)+s_j\left(f(B)g(B)\right)&\nonumber\geq s_j\left( s_n\left(f(A)\right)g(B)+s_j\left(g(A)\right)f(B)\right)\\&\geq s_n\left(s_n\left(f(A)g(B)\right)\right)+s_j\left(s_j\left(g(A)\right)f(B)\right)\nonumber\\
 &\qquad\qquad \qquad\qquad\qquad \qquad \quad(\textrm{by inequality (\ref{m13}) \nonumber })\\&=s_n\left(f(A)\right)s_n\left(g(B)\right)+s_j\left(g(A)\right)s_j\left(f(B)\right)\,\,(1\leq j\leq n).
\end{align}
In inequality (\ref{m10}), if we interchange the roles of $A$ and $B$, then we get
\begin{align}\label{m11}
 s_j\left(f(B)g(B)\right)+s_j\left(f(A)g(A)\right)\geq s_n\left(f(B)\right)s_n\left(g(A)\right)+s_j\left(g(B)\right)s_j\left(f(A)\right)\,\,(1\leq j\leq n).
\end{align}
By (\ref{m10}) and (\ref{m11})
\begin{align*}\
s_j\left(f(A)g(A)\right)+s_j\left(f(B)g(B)\right)&\geq
s_n\left(f(A)\right)s_n\left(g(B)\right)\\&\quad+{1\over2}\left(
s_j\left(g(A)\right)s_j\left(f(B)\right)+s_j\left(g(B)\right)s_j\left(f(A)\right)\right)
\end{align*}
\end{proof}
for all $1\leq j\leq n$.

In the following example we show that the constant ${1\over 2}$ is
the best possible one.

\begin{example}
For arbitrary synchronous functions $f,g:[0,+\infty)\rightarrow
{[0,+\infty)}$, let us put $A=B=I_{n\times n}$. Then
$s_j\left(f(A)g(B)\right)=s_j\left(f(B)g(B)\right)=f(1)g(1)$ and
$s_j\left(f(B)g(A)\right)=s_j\left(g(B)f((A)\right)=f(1)g(1),\hspace{.1cm}(1\leq
j\leq n)$. Thus
\begin{align*}\
s_j\left(f(A)g(A)\right)+s_j\left(f(B)g(B)\right)&=
s_n\left(f(A)\right)s_n\left(g(B)\right)\\&\quad+{1\over2}\left(
s_j\left(g(A)\right)s_j\left(f(B)\right)+s_j\left(g(B)\right)s_j\left(f(A)\right)\right)
\end{align*}
for all $j= 1, 2, \cdots, n$.
\end{example}
Using the same strategy as in the proof of Theorem \ref{36} we get
the next result.
\begin{theorem}
Let $f,g:[0,+\infty)\rightarrow {[0,+\infty)}$ be synchronous
functions. Then
\begin{eqnarray*}
f\left(s_j(A)\right)g\left(s_j(A)\right)+s_j\left(f(B)g(B)\right)\geq
f\left(s_j(A)\right)s_n\left(g(B)\right)+s_j\left(f(B)\right)g\left(s_j(A)\right)
\end{eqnarray*}
for all positive matrices $A,B\in\mathbb{M}_n$ and for all $j=1, 2,
\cdots, n$.
\end{theorem}

\begin{example}
Let $A, B$ be positive $n\times n$ matrices and $p,q>0$. Then
\begin{align*}\
s_j(A)^ps_j(A)^q+ s_j(B^{p+q}) \geq s_j(A)^{p}s_n(B^{q})+s_j(B^p)s_j(A)^q \hspace{.2cm}(1\leq j\leq n).
\end{align*}
\end{example}

\section*{Acknowledgment}
The second author would like to thank Tusi Mathematical Research Group (TMRG).


\bibliographystyle{amsplain}

\end{document}